\documentclass[10pt]{amsart}

\usepackage{amsmath, amsfonts, amsthm, amssymb, graphicx, fullpage, enumerate, subcaption, float, hyperref}
\usepackage{scalerel}
\usepackage{boldline} 
\usepackage{makecell} 
\usepackage{longtable} 
\usepackage{array}  

\newtheorem{theorem}{Theorem}[section]        
\newtheorem{lemma}[theorem]{Lemma}
\newtheorem{corollary}[theorem]{Corollary}

\newtheorem{proposition}[theorem]{Proposition}
    
\newtheorem*{main theorem}{Main Theorem}

\theoremstyle{remark}      
\newtheorem*{rem}{Remark}   
\theoremstyle{definition}  
\newtheorem{definition}[theorem]{Definition} 
\newtheorem{example}[theorem]{Example}  
      
\def\N{\mathbb{N}}

\def\Z{\mathbb{Z}}

\def\iff{\Longleftrightarrow}

\begin{document}
\title{The Pigeonhole Principle and Multicolor Ramsey Numbers} 
\author{Vishal Balaji, \quad Powers Lamb, \quad Andrew Lott, \quad Dhruv Patel,\\ Alex Rice, \quad Sakshi Singh, \quad Christine Rose Ward}
 
\begin{abstract} For integers $k,r\geq 2$, the diagonal Ramsey number $R_r(k)$ is the minimum $N\in\N$ such that every $r$-coloring of the edges of a complete graph on $N$ vertices yields on a monochromatic subgraph on $k$ vertices. Here we make a careful effort of extracting explicit upper bounds for $R_r(k)$ from the pigeonhole principle alone. Our main term improves on previously documented explicit bounds for $r\geq 3$, and we also consider an often ignored secondary term, which allows us to subtract a uniformly bounded below positive proportion of the main term. Asymptotically, we give a self-contained proof that $R_r(k)\leq \left(\frac{3+e}{2}\right)\frac{(r(k-2))!}{\left((k-2)!\right)^r}(1+o_{r\to \infty}(1)),$ and we conclude by noting that our methods combine with previous estimates on $R_r(3)$ to improve the constant $\frac{3+e}{2}$ to $\frac{3+e}{2}-\frac{d}{48}$, where $d=66-R_4(3)\geq 4$. We also compare our formulas, and previously documented formulas, to some collected numerical data.

\end{abstract}

\address{Department of Mathematics, Millsaps College, Jackson, MS 39210}
\email{balajv1@millsaps.edu} 
\email{lambps@millsaps.edu} 
\email{lottal@millsaps.edu}
\email{pateldu@millsaps.edu}    
\email{riceaj@millsaps.edu}
\email{singhs@millsaps.edu} 
\email{wardcr@millsaps.edu} 

\maketitle 
\setlength{\parskip}{5pt}   

\section{Introduction}

The classic \textit{party problem} poses the following question: how many people must be at a party in order to guarantee that there is either a group of three mutual acquaintances, or a group of three mutual strangers? A slightly impractical assumption here is that any two people are either acquainted or not, with no ambiguity. This question can be modeled with graphs, with each party guest represented by a vertex, and edges connecting each pair of guests colored red or blue according to their acquaintance status. With this framing, the desired occurrence is nicely characterized as a triangle formed by edges of all the same color, which we refer to as \textit{monochromatic}. Such an occurrence is not guaranteed amongst five partygoers. Consider the five vertices of a regular pentagon, with edges between consecutive vertices colored blue, and edges between nonconsecutive vertices colored red. Every triangle contains at least one red edge and at least one blue edge. However, with six people, trial and error suggests that such pattern avoidance is impossible. Indeed this is the case, and the standard argument is as follows. 

Consider a party with six people, and consider one particular guest; call them Jordan. Jordan can divide the remaining five people into two groups: those they know, and those they do not. The larger of these groups must have at least three people. This is an application of the pigeonhole principle, with the alternative being that each group has at most two people, in which case the total number of non-Jordan partygoers is at most four. Whichever group has at least three people, let that be the blue acquaintance status, and let its opposite be red. We now have the Jordan vertex connected to three other vertices with blue edges. If there is a blue edge between a pair of these three vertices, then that pair forms a blue triangle with Jordan. Otherwise, the three vertices form a red triangle amongst themselves. In particular, a monochromatic triangle exists no matter what, which completes the proof.

This problem generalizes very naturally. Why just a group of three? Why just two colors? Indeed, the classic party problem is the first nontrivial case of the following general result due to Ramsey \cite{ramsey}, which launched a robust area of combinatorial research. Recall that a \textit{complete graph} is one in which every pair of vertices is connected with an edge, and for $r\in \N$ we refer to a partition of the edges of a complete graph into $r$ pairwise disjoint sets (which we think of as $r$ different colors) as an \textit{$r$-coloring}. 
 
\begin{theorem}[Ramsey's Theorem]\label{ramthm} For $r,k\in \N$, there exists $N\in \N$ such that every $r$-coloring of  the edges of a complete graph on $N$ vertices yields a monochromatic complete subgraph on $k$ vertices. 
\end{theorem} 

As stated, Theorem \ref{ramthm} is purely qualitative. However, the existence of $N\in \N$ that satisfies the conclusion implies the existence of a \textit{first} $N\in \N$ that satisfies the conclusion, in other words a ``breaking point" at which the desired pattern switches from non-guaranteed to guaranteed. 

The following definition captures this breaking point, allowing for an arbitrary number of colors, and also allowing for different desired pattern sizes in each color. 

\begin{definition} For $r,k_1,\dots,k_r\in \N$, we define the \textit{Ramsey number} $R_r(k_1,\dots,k_r)$ to be the minimum $N\in \N$ such that every $r$-coloring of the edges of  a complete graph on $N$ vertices yields a monochromatic complete subgraph on $k_i$ vertices in color $i$ for some $1\leq i \leq r$. For $r,k\in \N$, we define the \textit{diagonal Ramsey number} $R_r(k)=R_r(k,\dots,k)$.  

\end{definition}

Armed with this notation, the solution to the original party problem can be summarized as $R_2(3)=6$. In this paper, we attempt a thorough, self-contained investigation of the extent to which explicit upper bounds on Ramsey numbers can be extracted from the pigeonhole principle; in other words an appropriate adaptation of the argument outlined in the opening paragraphs. For an appetizer of sorts, a special case of our results is as follows, in which $o(1)$ denotes an unspecified function tending to $0$ as $r\to \infty$. 

\begin{theorem}\label{appetizer} For an integer $r\geq 2$,  \begin{equation*}R_r(4)\leq \frac{(2r)!}{2^r}\left(\frac{10}{3}-\frac{2(r-2)}{(2r-1)(2r-3)} \right),  \end{equation*} and asymptotically \begin{equation*}R_r(4)\leq \frac{(2r)!}{2^r}\left(\frac{3+e}{2} \right)(1+o(1)). \end{equation*} 
\end{theorem}

\noindent For those interested in skipping ahead, the first component of Theorem \ref{appetizer} follows from Proposition \ref{diff}, Theorem \ref{main}, and Theorem \ref{wasterec}, while the asymptotic upper bound is a special case of Corollary \ref{asymcor}.

In Section \ref{prelim}, we establish the appropriate basics for the application of the pigeonhole principle to multicolor Ramsey numbers, and in the process we briefly survey some classical results. In Section \ref{mainsec}, we introduce a decomposition of our efforts into a \textit{main term}, $M_r$, and a \textit{waste function}, $w_r$, the latter so named because it captures a beneficial component of the recursive upper bound yielded by the pigeonhole principle that is ignored in some previous explicit bounds. The decomposition takes the form $R_r\leq M_r -w_r$, so our attention is turned to upper bounds on $M_r$ and lower bounds for $w_r$, and our main results of this type are stated in this section. Section \ref{mtsec} is dedicated to the verification of our claims for $M_r$, while Section \ref{wastesec} focuses on $w_r$. We take a computational angle on our efforts in Section \ref{compsec}, and conclude in Section \ref{conc} with a discussion of how our results are compatible with more refined known estimates on $R_r(3)$. 

\section{Preliminaries and classical bounds} \label{prelim}

We begin this section with our most elementary and most central ingredient: a version of the pigeonhole principle in which a collection of elements are distributed into several boxes, with potentially different quotas for each box. This should be thought of as a generalized justification of the fact that when Jordan divided five people into two groups, the larger of the two groups had to have at least three people. Here and throughout the paper, we use $|X|$ to denote the number of elements in a finite set $X$.

\begin{proposition}[Asymmetric pigeonhole principle] \label{APHP} Suppose $n,r,m_1,\dots,m_r\in \N$, and suppose $A_1,\dots,A_r$ are sets with $|A_1\cup\cdots\cup A_r|=n$. If $n>m_1+\cdots+m_r-r$, then $|A_i|\geq m_i$ for some $1\leq i \leq r$.
\end{proposition}

\begin{proof} Suppose $n,r,m_1,\dots,m_r\in \N$, and suppose $A_1,\dots,A_r$ are sets with $|A_1\cup\dots\cup A_r|=n. $ Proceeding contrapositively, suppose that $|A_i|<m_i$ for all $1\leq i \leq r$. Since $m_i$ is an integer, $|A_i|<m_i$ implies $|A_i|\leq m_i-1$ for all $1\leq i \leq r$. Therefore, 
\begin{equation*}n=|A_1\cup\cdots\cup A_r| \leq |A_1|+\cdots+|A_r| \leq (m_1-1)+\cdots+(m_r-1)=m_1+\cdots+m_r-r,
\end{equation*} and the contrapositive is established. \end{proof}

 We now apply Proposition \ref{APHP} to obtain a recursive upper bound on multicolor Ramsey numbers, which appears implicitly in a paper of Greenwood and Gleason \cite{GG} and appears as Proposition 2.18 in \cite{GR}. This is a generalization of the process in which Jordan divided their fellow partygoers into two groups, and then we focusd on the larger group. The final case analysis in the proof below is analogous to the case analysis we conducted on the three vertices connected to Jordan with blue edges.  We use $K_n$ to denote the complete graph on $n$ vertices.

\begin{proposition}\label{recurs} For integers $r,k_1,\dots,k_r\geq 2$, \begin{equation*} \label{phpr} R_r(k_1,\dots,k_r)\leq R_r(k_1-1,k_2,\dots,k_r)+\cdots+R_r(k_1,k_2,\dots,k_r-1)-(r-2). \end{equation*} \end{proposition}

\begin{proof}Fix integers $r,k_1,\dots,k_r\geq 2$ and let $m_i=R_r(k_1,\dots,k_i-1,\dots,k_r)$ for each $1\leq i \leq r$. Let $N=m_1+\cdots+m_r-(r-2)$. To prove that $R_r(k_1,\dots,k_r)\leq N$, we must show that every $r$-coloring of $K_N$ yields a monochromatic $K_{k_i}$ subgraph in color $i$ for some $1\leq i \leq r$. To this end, fix an $r$-coloring of $K_N$, fix a particular vertex $v$, and group the remaining $N-1$ vertices into sets $A_1,\dots,A_r$ based on the color of the edge connecting them with $v$.  Since $N-1>m_1+\cdots+m_r-r$, Proposition \ref{APHP} yields that $|A_i|\geq m_i$ for some $1\leq i \leq r$. In particular, amongst the vertices in $A_i$, there is either a monochromatic $K_{k_i-1}$ subgraph in color $i$, or there is a monochromatic $K_{k_j}$ subgraph in color $j$ for some $j\neq i$. In the latter case, we are done. In the former case, because the edges connecting $v$ to the vertices of $A_i$ are all color $i$, the set of vertices $A_i\cup\{v\}$ determines a monochromatic $K_{k_i}$ subgraph in color $i$, which completes the proof. 
\end{proof}

In the case $r=2$, Proposition \ref{recurs} yields an inequality that is reminiscent of the Pascal's triangle relation obeyed by binomial coefficients. This observation can be made precise with the following explicit upper bound on classical Ramsey numbers, first established by Erd\H{o}s and Szekeres \cite{ES}.

\begin{proposition}\label{twocolors} For $k_1,k_2\in \N$, \begin{equation*} R_2(k_1,k_2)\leq {k_1+k_2-2 \choose k_1-1}. \end{equation*}

\end{proposition}

\begin{proof} We induct on $k_1+k_2$. We first note that for $k_1=k_2=1$, $R_2(1,1)=1$ as any single vertex trivially constitutes a monochromatic $K_1$ subgraph in every color. On the right hand side, we have ${1+1-2 \choose 1-1}={0 \choose 0}=1$, and the base case is established.

\noindent Now suppose the conclusion holds for all $k_1,k_2\in \N$ summing to a particular $s\geq 2$, and fix $k_1,k_2\in \N$ with $k_1+k_2=s+1$. By Proposition \ref{recurs}, we have $R_2(k_1,k_2)\leq R_2(k_1-1,k_2)+R_2(k_1,k_2-1)$, and since $(k_1-1)+k_2=k_1+(k_2-1)=s$, we can invoke the inductive hypothesis for each term on the right hand side. Pairing this with the usual Pascal's triangle relation ${n\choose k}={n-1 \choose k-1} + {n-1 \choose k}$, we have 
\begin{align*} R_2(k_1,k_2)&\leq R_2(k_1-1,k_2)+R_2(k_1,k_2-1) \\ &\leq {(k_1-1)+k_2-2 \choose (k_1-1)-1} + {k_1+(k_2-1)-2 \choose k_1-1} \\ &={k_1+k_2-2 \choose k_1-1},
\end{align*}
as required.
\end{proof}

In order to generalize Proposition \ref{twocolors} to more than two colors, it is natural to consider whether Proposition \ref{recurs} can be used in conjunction with a Pascal's triangle-like relation on multinomial coefficients, which leads us toward the following definition and proposition.

\begin{definition} For $n,r \in \N$ and integers $k_1,\dots,k_r\geq 0$ with $k_1+\cdots+k_r=n$, we define the \textit{multinomial coefficient} ${n \choose k_1,\dots,k_r}$  to be the number of (ordered) partitions of an $n$-element set $A$ into $A_1\cup \cdots \cup A_r$ with $|A_i|=k_i$ for all $1\leq i \leq r$. Equivalently, ${n \choose k_1,\dots,k_r}$ is the coefficient on $x_1^{k_1}\cdots x_r^{k_r}$ in the expansion of $(x_1+\cdots+x_r)^n$.
\end{definition}

We include proofs of the following standard formula and recurrence relation for completeness.
 
\begin{proposition}\label{multin} For $n,r \in \N$ and integers $k_1,\dots,k_r\geq 0$ with $k_1+\cdots+k_r=n$, \begin{equation*} {n \choose k_1,\dots,k_r}=\frac{n!}{k_1!\cdots k_r!}. \end{equation*} Further, if $k_1,\dots,k_r>0$, then \begin{equation*}{n \choose k_1,\dots,k_r}= {n-1 \choose k_1-1,k_2,\dots,k_r}+\cdots +{n-1 \choose k_1,k_2\dots,k_r-1}. \end{equation*}
\end{proposition}
 
\begin{proof} Fix $n,r \in \N$ and integers $k_1,\dots,k_r\geq 0$ with $k_1+\cdots+k_r=n$. Let $A$ be any set with $|A|=n$. For the first formula, we begin by noting that the total number of orderings of the form $A=\{a_1,a_2,\dots,a_n\}$ is $n!$. Further, every such ordering defines a partition of $A$ into $A=A_1\cup\cdots\cup A_r$ with $|A_i|=k_i$ for all $1\leq i \leq r$ by letting $A_1=\{a_1,\dots,a_{k_1}\}, \ A_2=\{a_{k_1+1},\dots,a_{k_1+k_2}\},\dots,A_r=\{a_{k_1+\cdots+k_{r-1}+1},\dots,a_{n}\}$. However, while the multinomial coefficient accounts for the ordering of the \textit{sets} $A_1,\dots,A_r$, each individual partition does not care about the order in which the \textit{elements} of the sets $A_1,\dots,A_r$ are listed. In particular, an individual partition $A=A_1\cup \cdots \cup A_r$ is counted once for every possible choice of the ordering of the elements of each set, a total of $k_1!k_2!\cdots k_r!$ times. Therefore, $ {n \choose k_1,\dots,k_r}=\frac{n!}{k_1!\cdots k_r!},$ as claimed.

\noindent For the recursive relation, a generalized form of the usual Pascal's triangle relation, we argue both combinatorially and algebraically. Fix a particular $x\in A$. Consider a partition $A=A_1\cup\cdots\cup A_r$ with $|A_i|=k_i$ for all $1\leq i \leq r$. If $x$ is assigned to $A_i$, the remaining $n-1$ elements must be distributed among $A_1,\dots, A_r$, with exactly $k_i-1$ elements in $A_i$ and exactly $k_j$ elements in $A_j$ for all $j\neq i$. In other words, the number of qualifying partitions with $x\in A_i$ is precisely the multinomial coefficient $ {n-1 \choose k_1,\dots,k_i-1,\dots,k_r}$. Since $x$ must be assigned to $A_i$ for some $1\leq i \leq r$, we have $$ {n \choose k_1,\dots,k_r}=\sum_{i=1}^r {n-1 \choose k_1,\dots,k_i-1,\dots,k_r}, $$ as claimed. Alternatively, we can use the previously established formula to observe 
\begin{align*}{n-1 \choose k_1-1,k_2,\dots,k_r}+\cdots +{n-1 \choose k_1,k_2,\dots,k_r-1}
& = \frac{(n-1)!}{(k_1 - 1)!k_2!\cdots k_r!} +  \cdots + \frac{(n-1)!}{k_1!k_2!\cdots (k_r-1)!}
\\  & =\frac{k_1(n-1)!}{k_1!\cdots k_r!} + \cdots + \frac{k_r(n-1)!}{k_1!\cdots k_r!}\\
& =\frac{(k_1 + \cdots + k_r)(n-1)!}{k_1!\cdots k_r!} \\
& =\frac{n(n-1)!}{k_1!\cdots k_r!}=\frac{n!}{k_1!\cdots k_r!}={n \choose k_1, \dots,k_r}.  \qedhere \end{align*}
 \end{proof}

A natural candidate for a generalization of Proposition \ref{twocolors} using Proposition \ref{multin} is the bound \begin{equation} \label{cbound} R_r(k_1,\dots,k_r)\leq {k_1+\cdots+k_r-r \choose k_1-1,\dots,k_r-1}. \end{equation} Indeed \eqref{cbound} holds, as seen by a straightforward adaptation of the induction proof of Proposition \ref{twocolors} above, and it appears in numerous sources in the literature, including Corollary 3 of \cite{GG}. However, for $r\geq 3$, \eqref{cbound} does a surprisingly poor job of capturing upper bounds on Ramsey numbers available from Proposition \ref{recurs}, as we will explore with examples later. One explanation of this phenomenon is the fact that, for $r\geq 3$, the right side of \eqref{cbound} lacks a fundamental property that the left side possesses. Namely, Ramsey numbers ``ignore $2$'s", by which we mean the following.

\begin{proposition}\label{ignore2s} For integers $r,k_1,\dots,k_{r-1}\geq 2$  \begin{equation*} \label{ignore2} R_r(k_1,\dots,k_{r-1},2)= R_{r-1}(k_1,\dots,k_{r-1}). \end{equation*} \end{proposition}

\begin{proof} Suppose $r,k_1,\dots,k_{r-1}\geq 2$ are integers. Let $N=R_{r-1}(k_1,\dots,k_{r-1})$. An $(r-1)$-coloring of $K_{N-1}$ with no $K_{k_i}$ subgraph in color $i$ for all $1\leq i \leq r-1$ also counts as a valid $r$-coloring of $K_{N-1}$ that simply does not use color $r$, so in particular there are no $K_2$ subgraphs in color $r$. Therefore, $R_r(k_1,\dots,k_{r-1},2)\geq N$. Conversely, for any $r$-coloring of $K_N$, since a $K_2$ subgraph consists of a single edge, the \textit{only} way to avoid a $K_2$ subgraph in color $r$ is to not use color $r$ at all. In this case, we in fact have an $(r-1)$-coloring of $K_N$, which must yield a $K_{k_i}$ subgraph in color $i$ for some $1\leq i \leq r-1$. Therefore, $R_r(k_1,\dots,k_{r-1},2)\leq N$, which completes the proof.\end{proof}

To begin to illustrate our assertion that \eqref{cbound} does not fully capture the capability of the pigeonhole principle in bounding multicolor Ramsey numbers, we consider the following example.

\begin{example}\label{r34} $R_3(4)\leq 272$.

\noindent Repeatedly applying Propositions \ref{recurs}, \ref{ignore2s}, and \ref{twocolors}, and the fact that Ramsey numbers are invariant under coordinate permutation, we see
\begin{align*}R_3(4)&=R_3(4,4,4) \\ &\leq 3R_3(4,4,3)-1 \\ &\leq 6R_3(4,3,3)+ 3R_2(4,4) - 4 \\ & \leq 6R_3(3,3,3)+12R_2(4,3)+3R_2(4,4)-10 \\ & \leq 18R_3(3,3)+12R_2(4,3)+3R_2(4,4)-16 \\ & \leq 18(6)+12(10)+3(20)-16 \\ &=272.
\end{align*}
\end{example}

\begin{rem} In fact it is known that $R_2(4,3)=9$, not $10$ as implied by $R_2(4,3) \leq R_2(3,3) + R_2(4,2)$, due to the fact that the total degree of a graph must be an even number. This consideration generalizes to show that, when $r=2$, the inequality in Proposition \ref{twocolors} is strict if both Ramsey numbers on the right hand side are even, another result that dates back to \cite{GG}, as is the case with $R_2(4,2)=4$ and $R_2(3,3)=6$. This fact, and any other mitigating considerations, can be used when applying Proposition \ref{recurs} to yield improved upper bounds, in this case $R_3(4) \leq 260$. However, in this paper we focus our attention on capturing the strength of the pigeonhole principle alone. 
\end{rem}

In contrast with Example \ref{r34}, plugging in $r=3$ and $k_1=k_2=k_3=4$ to \eqref{cbound} yields $R_3(4)\leq \frac{9!}{3!^3}= 1680$. The spirit of this observation is not new, for example Graham and R\"odl \cite{GR} remark that \eqref{cbound} ``can easily be improved by a factor that tends to $0$ as $r\to \infty$." However, documented examples of this have been difficult to locate in the literature. One example, which helps alleviate the discrepancy between how Ramsey numbers and \eqref{cbound} handle inputs equal to $2$, appears in an unpublished note of Ter\"av\"ainen \cite{joni}. In that note, Proposition \ref{twocolors} is used as a base case for an induction on $r$, and then Proposition \ref{ignore2s} is implicitly used to establish \begin{equation}\label{jonib} R_r(k_1,\dots,k_r)\leq {k_1+\cdots+k_r-2r+2 \choose k_1-1,k_2-1,k_3-2,\dots,k_r-2} \end{equation} for $r,k_1,\dots,k_r\geq 2$. To again compare with Example \ref{r34}, plugging in $r=3$ and $k_1=k_2=k_3=4$ into \eqref{jonib} yields $R_3(4)\leq 560$. In the diagonal case, \eqref{jonib} is asymptotically smaller than \eqref{cbound} by a factor of at least $(r/e)^{r-2}$ as $r\to \infty$. One of our main results, Theorem \ref{main}, can be thought of as achieving the same goal as \eqref{jonib}, the establishment of a multinomial coefficient upper bound that respects Proposition \ref{ignore2s}, but in a more efficient way.

In addition to its lack of exploitation of Proposition \ref{ignore2s}, another potential weakness of \eqref{cbound}, which is shared by \eqref{jonib}, is that it fails to take any advantage of the $r-2$ term in Proposition \ref{recurs}, which works in our favor for $r\geq 3$. An example of this  term in action is the following, which is essentially Theorem 2.20 in \cite{GR}. Here and for the remainder of the paper we use $e_r=\sum_{n=0}^r 1/n!$ to denote the $r$-th partial sum in the Taylor expansion of $e$. 
  
\begin{proposition}\label{r3} For $r\in \N$, $R_r(3)\leq e_rr!+1$.

\end{proposition}

\begin{proof} We induct on $r$. For $r=1$, $R_1(3)$ is the number of vertices required to guarantee a monochromatic triangle in any $1$-coloring, which is $3$. For the right hand side, we have $e_1(1!)+1=(1+1)(1)+1=3$, so the base case is established.

\noindent Now assume the conclusion holds for some particular $r\in \N$. By Proposition \ref{recurs}, the invariance of Ramsey numbers under coordinate permutation, and Proposition \ref{ignore2s}, we have 
\begin{align*}R_{r+1}(3)&\leq R_{r+1}(2,3,\dots,3)+\cdots +R_{r+1}(3,\dots,3,2)-(r+1-2) \\&= (r+1)R_{r+1}(3,\dots,3,2)-(r-1) \\ &=(r+1)R_r(3)-(r-1) \\ &\leq (r+1)(e_rr!+1)-(r-1) \\ &=e_r(r+1)!+1+1 \\ &=\left(e_r+\frac{1}{(r+1)!}\right)(r+1)!+1 \\ &=e_{r+1}(r+1)!+1,
\end{align*} and the induction is complete.
\end{proof}

It is important to note that in our brief survey here, we have focused on only the most elementary tools, and have covered only a tiny portion of the techniques and extensive literature that have developed on Ramsey numbers over the past near century. For an impressively comprehensive treatment, the interested reader is encouraged to browse the referenced survey of Radziszowski \cite{Radz}, updated most recently in January 2021. Some notable examples include improved bounds on diagonal two-color Ramsey numbers by Conlon \cite{conlon} and Sah \cite{sah}, the latter of whom showed $R_2(k)\leq {2k-2 \choose k-1} e^{-c(\log k)^2}$ for a constant $c>0$, upper bounds on $R_r(3)$ due to Xu-Xie-Chen \cite{Xu} and Eliahou \cite{eli} that we discuss in Section \ref{conc}, as well as a great deal of work on lower bounds, modifications for incomplete graphs, hypergraphs, and more.  

\section{Main Results} \label{mainsec}

As indicated in the introduction, we divide the task of efficiently capturing the strength of Proposition \ref{recurs} into an explicit upper bound  on multicolor Ramsey numbers into two components: a \textit{main term}, $M_r$, which accounts for the Ramsey number summands on the right hand side of the inequality in Proposition \ref{recurs}, and a \textit{waste function}, $w_r$, which accounts only for the subtracted $(r-2)$ term. 

\begin{definition}\label{wastedef} We recursively define functions $M_r, w_r:\Z_{\geq2}^r\to \Z_{\geq 0}$ for integers $r\geq 2$ as follows:
 
\begin{enumerate}[(i)] \item $M_2(k_1,k_2)={k_1+k_2-2 \choose k_1-1}$ and $w_2(k_1,k_2)=0$ for all $k_1,k_2\geq 2$

\

\item Both $M_r$ and $w_r$ ``ignore $2$'s", meaning $$M_{r+1}(2,k_1,\dots,k_{r})=M_{r+1}(k_1,2,k_2,\dots,k_r)=\cdots=M_{r+1}(k_1,\dots,k_r,2)=M_{r}(k_1,\dots,k_{r})$$ and $$w_{r+1}(2,k_1,\dots,k_{r})=w_{r+1}(k_1,2,k_2,\dots,k_r)=\cdots=w_{r+1}(k_1,\dots,k_r,2)=w_{r}(k_1,\dots,k_{r})$$ for all $k_1,\dots,k_r\geq 2$

\

\item If $k_1,\dots,k_r\geq 3$, then \begin{equation*}M_r(k_1,\dots,k_r)=M_r(k_1-1,\dots,k_r)+\cdots+M_r(k_1,\dots,k_r-1). \end{equation*} and \begin{equation*}w_r(k_1,\dots,k_r)=w_r(k_1-1,\dots,k_r)+\cdots+w_r(k_1,\dots,k_r-1)+(r-2). \end{equation*}

\

\noindent In Proposition \ref{diff}, we prove via double induction on $r$ and $k_1+\cdots+k_r$ that $M_r$ and $w_r$ are uniquely determined by properties (i)-(iii), and also we have:

\

\item $M_r$ and $w_r$ are invariant under coordinate permutation, just like Ramsey numbers. 

\end{enumerate}   
Also, for a single integer $k\geq 2$, we define $M_r(k)=M_r(k,\dots,k)$ and  $w_r(k)=w_r(k,\dots,k)$.

\end{definition}

Before getting into the weeds with these two components, we first establish that upper bounds for multicolor Ramsey numbers are implied by upper bounds for $M_r$ and lower bounds for $w_r$.

\begin{proposition}\label{diff} The functions $M_r, w_r:\Z_{\geq2}^r\to \Z_{\geq 0}$ for integers $r\geq 2$ are uniquely determined by properties (i)-(iii) in Definition \ref{wastedef}, and are invariant under coordinate permutation. Further, for integers $r,k_1,\dots,k_r\geq 2$, \begin{equation*}R_r(k_1,\dots,k_r)\leq M_r(k_1,\dots,k_r)-w_r(k_1,\dots,k_r). \end{equation*}
\end{proposition}

\begin{proof} We prove all three claims with a double induction, an ``outer induction" on the number of colors, $r$, and an ``inner induction" on $k_1+\cdots+k_r$. 

\noindent For the outer base case $r=2$, we have that $M_2(k_1,k_2)={k_1+k_2-2 \choose k_1-1}$ and $w_2(k_1,k_2)=0$ are uniquely determined and invariant under coordinate permutation. Further, the inequality $$R_2(k_1,k_2)\leq M_2(k_1,k_2)-w_2(k_1,k_2)={k_1+k_2-2 \choose k_1-1}$$ is precisely Proposition \ref{twocolors}, and the outer base case is fully established.

\noindent Now suppose all three claims hold for a particular $r\geq 2$ and all $k_1,\dots,k_r\geq 2$. We now verify that all three claims hold for $r+1$ and all $k_1,\dots,k_{r+1}\geq 2$ by induction on $k_1+\cdots+k_{r+1}$. By Proposition \ref{ignore2s} and property (ii) of Definition \ref{wastedef}, we know that all three claims hold for all $k_1,\dots,k_{r+1}$ if $\min\{k_1,\dots,k_{r+1}\}=2$, as this reduces back to the case of $r$ coordinates. In particular, this establishes the inner base case of $k_1=\cdots=k_{r+1}=2$, and allows us to assume moving forward that $k_1,\dots,k_{r+1}\geq 3$. We now assume all three claims hold for all $k_1,\dots,k_{r+1}\geq 2$ adding to a particular $s\geq 2(r+1)$, and we fix $k_1,\dots,k_{r+1} \geq 3$ with $k_1+\cdots+k_{r+1}=s+1$.

\noindent By property (iii), $M_{r+1}(k_1,\dots,k_{r+1})$ and $w_{r+1}(k_1,\dots,k_{r+1})$ are given by sums of values of $M_{r+1}$ and $w_{r+1}$, respectively, with coordinates adding to $s$ (plus a constant in the case of $w_{r+1}$). Since each of these are uniquely determined by inductive hypothesis, so are $M_{r+1}(k_1,\dots,k_{r+1})$ and $w_{r+1}(k_1,\dots,k_{r+1})$. Further, for any permutation $\phi:\{1,\dots,r+1\}\to \{1,\dots,r+1\}$, we have \begin{align*}M_{r+1}(k_{\phi(1)},\dots,k_{\phi(r+1)})&=M_{r+1}(k_{\phi(1)}-1,\dots,k_{\phi(r+1)})+\cdots+ M_{r+1}(k_{\phi(1)},\dots,k_{\phi(r+1)}-1) \\ &=M_{r+1}(k_1,\dots,k_{\phi(1)}-1,\dots,k_{r+1})+\cdots+M_{r+1}(k_1,\dots,k_{\phi(r+1)}-1,\dots,k_{r+1}) \\ &=M_{r+1}(k_1-1,\dots,k_r)+\cdots+M_{r+1}(k_1,\dots,k_{r+1}-1) \\ &=M_{r+1}(k_1,\dots,k_{r+1}),\end{align*} where the coordinates are permuted in the second line to put the indices in order, as allowed by inductive hypothesis, and the summands in the third line are rearranged to match property (iii). The reasoning for $w_{r+1}$ is identical.

\noindent Finally, by Proposition \ref{recurs}, property (iii), and inductive hypothesis, we have \begin{align*}R_{r+1}(k_1,\dots,k_{r+1})\leq& R_r(k_1-1,\dots,k_r)+\cdots+R_r(k_1,\dots,k_r-1)-(r-2) \\ \leq& M_r(k_1-1,\dots,k_r)-w_r(k_1-1,\dots,k_r)+\cdots \\ &+M_r(k_1,\dots,k_r-1)-w_r(k_1,\dots,k_r-1)-(r-2) \\ =&(M_r(k_1-1,\dots,k_r)+\cdots+M_r(k_1,\dots,k_r-1))\\ &-(w_r(k_1-1,\dots,k_r)+\cdots+w_r(k_1,\dots,k_r-1)+(r-2)) \\ =&M_{r+1}(k_1,\dots,k_r)-w_r(k_1,\dots,k_r),\end{align*} and all three claims are established. \end{proof}

As mentioned in Section \ref{prelim}, our first effort in wrestling with $M_r$, proven with the same double induction structure as Proposition \ref{diff}, is similar in spirit to \eqref{jonib}. The improvements in comparison to \eqref{jonib} are gained from a more careful inner induction step. 

\begin{theorem}\label{main} For integers $r\geq 2$ and $k_1\geq k_2\geq \dots \geq k_r\geq 2$,  \begin{equation}\label{genmain} M_r(k_1,\dots,k_r)\leq \frac{(k_1+k_2-2)(k_1+k_2-3)}{(k_1-1)(k_2-1)}  {k_1+\cdots+k_r-2r \choose k_1-2,\dots, k_r-2}. \end{equation}
In particular, for integers $r,k\geq 2$,   \begin{equation}\label{better} M_r(k)\leq \left(4-\frac{2}{(k-1)} \right)\frac{(r(k-2))!}{\left((k-2)!\right)^r}. \end{equation} 
\end{theorem}

\noindent The right hand side of \eqref{better} is smaller than the diagonal case of \eqref{jonib} by at least a factor of $\frac{r^2(k-2)^2}{2(2k-3)(k-1)}$. Returning to Example \ref{r34}, and the bounds $R_3(4)\leq 1680$ and $R_3(4)\leq 560$ yielded by \eqref{cbound} and \eqref{jonib}, respectively, plugging in $r=3$ and $k=4$ into \eqref{better} yields $R_3(4)\leq M_3(4)\leq 300$. Further inspection of Example \ref{r34} with our new bifurcated perspective yields the exact value $M_3(4)=288$, while $w_3(4)=16$, so the estimate $M_3(4)\leq 300$ is relatively efficient.

 Further, if in Example \ref{r34} we were to continue applying Proposition \ref{recurs} even after reducing to $2$ coordinates, we would eventually be left with  $18=(3)(2)(3)$ copies of $R_2(4,2)=4$ and $72=(3)(2)(12)$ copies of $R_2(3,2)=3$, which yields $M_3(4)=(18)(4)+(72)(3)=288$. The $(3)(2)$ represents the number of ordered choices of two coordinates, and the final factors of $3$ and $12$, respectively, arise as the number of paths, reducing a coordinate by $1$ each step, from $(4,4,4)$ to $(4,3,2)$ and $(3,3,2)$, respectively. At the expense of the compact, convenient formula provided by Theorem \ref{main}, this observation generalizes to the following exact formula for $M_r$. 

\begin{theorem} \label{exact} For integers $r\geq 2$ and $k_1,\dots,k_r\geq 3$, 
\begin{equation*}M_r(k_1,\dots,k_r)=\sum_{i=1}^r \sum_{\substack{j=1 \\ j\neq i}}^r \sum_{m=3}^{k_j} m {k_1+\cdots+k_r-2r-m+1 \choose k_1-2,\dots,k_i-3,\dots,k_j-m,\dots, k_r-2}.\end{equation*} In particular, for integers $r\geq 2$ and $k\geq 3$,
\begin{equation*}M_r(k)=r(r-1)\sum_{m=3}^k m {r(k-2)-m+1 \choose \underbrace{k-2,\dots, k-2}_{r-2},k-3,k-m}\end{equation*}
\end{theorem}

\noindent Theorem \ref{exact}, or a separate double induction proof, gives the following lower bound on $M_r$, which in particular assures that the convenient formula in Theorem \ref{main} is within a factor of $4/3$ of the true value of $M_r$ in the diagonal case.

\begin{corollary}\label{lbcor} For integers $r\geq 2$ and $k_1,\dots,k_r\geq 3$, \begin{equation*}M_r(k_1,\dots,k_r)\geq 3{k_1+\cdots+k_r-2r \choose k_1-2,\dots,k_r-2}. \end{equation*}
\end{corollary}

\noindent Combining Corollary \ref{lbcor} with an asymptotic upper bound on $M_r$ yielded by Theorem \ref{exact}, we establish the following asymptotic formula for $M_r(k)$. We use the notation $o_{r\to \infty}(1)$ to clarify that the asymptotic estimate is in the variable $r$, not $k$. 

\begin{corollary}\label{mainasym3} For integers $r\geq 2$ and $k\geq 3$, \begin{equation*}M_r(k)=3\frac{(r(k-2))!}{\left((k-2)!\right)^r}(1+o_{r\to \infty}(1)). \end{equation*}
\end{corollary}


Implicit in the proof of Proposition \ref{r3} (which, as noted, is not new) is the formula $w_r(3)=(3-e_r)r!-1$, the difference between the formula in the conclusion of Proposition \ref{r3} and the main term $3r!$ yielded by \eqref{better} when $k=3$. We include a separate proof of this formula for $w_r(3)$ in Section \ref{wastesec} for completeness, and then establish the following general lower bound in the diagonal case. 

\begin{theorem}\label{wasterec} For integers $r\geq 2$ and $k\geq 4$,  \begin{align*}w_r(k)& \geq (r-2)r!\left((e_r-1)\frac{r!^{k-3}-1}{r!-1}-r!^{k-4}\right) \\ &+ \sum_{j=0}^{r-3}{r\choose j}{r(k-3)+j-1 \choose \underbrace{k-2,\dots,k-2}_{j},\underbrace{k-3,\dots,k-3,k-4}_{r-j}}(r-j) \left[r-j-2+w_{r-j}(3)\right]. \end{align*}  
\end{theorem} 

\noindent The first upper bound in Theorem \ref{appetizer} from the introduction now follows from Theorem \ref{main} and the extraction of the $j=0$ term from the summation in Theorem \ref{wasterec}. Careful manipulation of the lower bound formula in Theorem \ref{wasterec}, and a healthy dose of calculus, gives the following asymptotic lower bound on $w_r(k)$, which notably has the same order of magnitude as $M_r(k)$. 

\begin{theorem}\label{wasteasym} For integers $r\geq 2$ and $k\geq 4$, 
\begin{equation*}w_r(k)\geq\left(\frac{3-e}{2}\right)\frac{(r(k-2))!}{\left((k-2)!\right)^r}\left(1-o_{r\to \infty}(1)\right).
\end{equation*} 
\end{theorem}

\noindent Finally, Theorems \ref{mainasym3} and \ref{wasteasym} combine with Proposition \ref{diff} to yield the following asymptotic upper bound for diagonal multicolor Ramsey numbers. 

\begin{corollary}\label{asymcor} For integers $r\geq 2$ and $k\geq 4$, \begin{equation*} R_r(k)\leq \left(\frac{3+e}{2}\right)\frac{(r(k-2))!}{\left((k-2)!\right)^r}\left(1+o_{r\to \infty}(1)\right).\end{equation*}
\end{corollary}

We have asserted throughout the paper that we seek explicit upper bounds on multicolor Ramsey numbers that capture the strength of the pigeonhole principle ``alone". However, as will be seen in the following two sections, some of our proofs use other tools, primarily calculus. To clarify our meaning, we note that, besides the basic properties $R_1(k)=k$, Proposition \ref{ignore2s}, and invariance under coordinate permutation, the recursive inequality in Proposition \ref{recurs} is the only thing we are using \textit{about Ramsey numbers}. Our efforts in Section \ref{mtsec} and \ref{wastesec} address the question of what explicit upper bounds must hold for any function satisfying that list of properties, and toward that end we incorporate a wider variety of techniques.  


\section{The main term} \label{mtsec}

In this section we verify all of our claims concerning $M_r$, namely Theorems \ref{main} and \ref{exact} and Corollaries \ref{lbcor} and \ref{mainasym3}. We begin with some preliminary observations related to the fraction multiplied by the multinomial coefficient in the conclusion of Theorem \ref{main}, which will be crucial in the proof of that theorem.

\begin{lemma} \label{alglemma} For $x,y>0$, let $f(x,y)=\frac{(x+y)(x+y-1)}{xy}$. Then, for $x,y>1$, $f(x-1,y)\leq f(x,y)$ if and only if $x \geq y$. Further, for integers $r,n_1,n_2\geq 2$ and $n_3,\dots,n_r\geq 0$, \begin{align*}&f(n_1-1,n_2){n_1+\cdots +n_r-3 \choose n_1-2,n_2-1,n_3\dots,n_r}+f(n_1,n_2-1){n_1+\cdots+n_r-3 \choose n_1-1,n_2-2,n_3,\dots,n_r} \\ =& f(n_1,n_2)\left( {n_1+\cdots+n_r-3 \choose n_1-2,n_2-1,n_3\dots,n_r}+{n_1+\cdots+n_r-3 \choose n_1-1,n_2-2,n_3,\dots,n_r}\right).\end{align*}
\end{lemma}

\begin{proof} For the first claim, we consider $x,y>1$ and see 
\begin{align*} f(x-1,y)\leq f(x,y) &\iff \frac{(x+y-1)(x+y-2)}{(x-1)y}  \leq \frac{(x+y)(x+y-1)}{xy} \\  & \iff x(x+y-2)  \le (x-1)(x+y) \\
& \iff x^2 + xy - 2x \leq x^2 + xy - x - y \\ &\iff x\geq y.
\end{align*} 

\noindent For the second claim, fix integers $r,n_1,\dots,n_r\geq 2$. We first apply the formula from Proposition \ref{multin}, and the two sides of the desired equality are \begin{equation*}f(n_1-1,n_2)\frac{(n_1+\cdots+n_r-3)!}{(n_1-2)!(n_2-1)!n_3!\cdots n_r!}+f(n_1,n_2-1)\frac{(n_1+\cdots+n_r-3)!}{(n_1-1)!(n_2-2)!n_3!\cdots n_r!} \end{equation*} and \begin{equation*}f(n_1,n_2)\left(\frac{(n_1+\cdots+n_r-3)!}{(n_1-2)!(n_2-1)!n_3!\cdots n_r!}+\frac{(n_1+\cdots+n_r-3)!}{(n_1-1)!(n_2-2)!n_3!\cdots n_r!}, \right) \end{equation*} respectively. Dividing both expressions by the common factor $$\frac{(n_1+\cdots+n_r-3)!}{(n_1-2)!(n_2-2)!n_3!\cdots n_r!}, $$ we have that the desired inequality holds if and only if \begin{equation*}\frac{f(n_1-1,n_2)}{n_2-1}+\frac{f(n_1,n_2-1)}{n_1-1}=f(n_1,n_2)\left(\frac{1}{n_2-1}+\frac{1}{n_1-1}\right), \end{equation*} or in other words \begin{equation*} \frac{(n_1+n_2-1)(n_1+n_2-2)}{(n_1-1)n_2(n_2-1)}+\frac{(n_1+n_2-1)(n_1+n_2-2)}{n_1(n_2-1)(n_1-1)}=\frac{(n_1+n_2)(n_1+n_2-1)}{n_1n_2}\left(\frac{1}{n_2-1}+\frac{1}{n_1-1}\right). \end{equation*} Dividing both sides by $(n_1+n_2-1)$ and multiplying both sides by $n_1n_2(n_1-1)(n_2-1)$,  this equality is equivalent to \begin{align*}(n_1+n_2-2)n_1+(n_1+n_2-2)n_2=(n_1+n_2)(n_1-1+n_2-1).
\end{align*} Finally, we see that both sides yield $(n_1+n_2)(n_1+n_2-2)$, and the identity is verified. \end{proof}

We now have the necessary facts to complete the double induction proof of our convenient, explicit upper bound on multicolor Ramsey numbers.

\begin{proof}[Proof of Theorem \ref{main}]  For integers $r,k_1,\dots,k_2\geq 2$, we let $b_{k_1,k_2}=f(k_1-1,k_2-1)$, where $f$ is as in Lemma \ref{alglemma}, and we let $C_r(k_1,\dots,k_r)={k_1+\cdots+k_r-2r \choose k_1-2,\dots, k_r-2}$. With this notation, we must show that \begin{equation}\label{goal} M_r(k_1,\dots,k_r)\leq b_{k_1,k_2}C_r(k_1,\dots,k_r) \end{equation} for all integers $r\geq 2$ and $k_1\geq k_2\geq \dots \geq k_r\geq 2$. We note that, by definition of $C_r$ and by Proposition \ref{multin}, $C_r$ satisfies the same properties as $M_r$ listed in (ii)-(iv) in Definition \ref{wasterec}, while $b_{k_1,k_2}$ is invariant under transposition of the first two coordinates but not  the full list of $r$ coordinates when $r\geq 3$. 

\noindent We establish \eqref{goal} with a double induction structured identically to that of the proof of Proposition \ref{diff}. For the outer base case $r=2$, we see for $k_1, k_2 \geq 2$ that \begin{align*}b_{k_1,k_2}C_2(k_1,k_2)&=\frac{(k_1+k_2-2)(k_1+k_2-3)}{(k_1-1)(k_2-1)}{k_1+k_2-4 \choose k_1-2,k_2-2} \\ &= \frac{(k_1+k_2-2)(k_1+k_2-3)}{(k_1-1)(k_2-1)}\frac{(k_1+k_2-4)!}{(k_1-2)!(k_2-2)!} \\ &=\frac{(k_1+k_2-2)!}{(k_1-1)!(k_2-1)!}\\ &= {k_1+k_2-2 \choose k_1-1}\\&= M_2(k_1,k_2).\end{align*} In fact, the fraction $b_{k_1,k_2}$ was chosen specifically to arrange this identity, and the outer base case is established. Now suppose \eqref{goal} holds for a particular $r\geq 2$ and all $k_1 \geq k_2\geq \cdots\geq k_r\geq 2$. By property (ii), we also know that \eqref{goal} holds for $k_1\geq \cdots \geq k_r\geq k_{r+1}=2$, as this reduces back to the case of $r$ coordinates. In particular, this establishes the inner base case of $k_1=\cdots=k_{r+1}=2$, and allows us to assume moving forward that $k_{r+1}\geq 3$. We now assume \eqref{goal} holds for all $k_1\geq \cdots \geq k_{r+1}\geq 2$ adding to a particular $s\geq 2(r+1)$, and we fix $k_1\geq \cdots \geq k_{r+1} \geq 3$ with $k_1+\cdots+k_{r+1}=s+1$.

\noindent We complete the inner induction step by analyzing three cases:

\noindent \textbf{Case 1:} $k_1=k_2$.

\noindent Suppose $k_1=k_2=k$. If $k_3$ is also $k$, when applying property (iii) and subtracting from the first two coordinates, we technically must permute coordinates to put them back into decreasing order to apply the inductive hypothesis, so the resulting first two coordinates may both be $k$ or may be one $k$ and one $k-1$. However, Lemma \ref{alglemma} comes to the rescue in this case, as it assures $b_{k,k}=b_{k-1,k}$, so $b_{k,k}$ is the correct coefficient on every term. Therefore, regardless of whether $k_3$ is equal to versus less than, $k$, we have by property (iii) and the inductive hypothesis that 

\begin{align*}M_{r+1}(k,k,k_3,\dots,k_{r+1})&=M_{r+1}(k-1,k,\dots,k_{r+1})+\cdots +M_{r+1}(k,k,\dots,k_{r+1}-1) \\ &\leq b_{k,k}\left(C_{r+1}(k-1,k,\dots,k_{r+1})+\cdots +C_{r+1}(k,k,\dots,k_{r+1}-1)\right)\\ & =b_{k,k}C_{r+1}(k,k,k_3,\dots,k_{r+1}).
\end{align*}  

\noindent \textbf{Case 2:} $k_1>k_2=k_3$.

\noindent Suppose $k_1>k_2=k_3=k$. Lemma \ref{alglemma} assures that $b_{k_1-1,k}<b_{k_1,k}$, and when applying property (iii) and subtracting from the second coordinate, the third coordinate must be moved to the second spot in order to apply the inductive hypothesis, so the correct coefficient on the second term is $b_{k_1,k}$. In other words, by property (iii) and inductive hypothesis, we have 

\begin{align*}M_{r+1}(k_1,k,k,k_4,\dots,k_{r+1})=&M_{r+1}(k_1-1,k,k,\dots,k_{r+1})+M_{r+1}(k_1,k-1,k\dots, k_{r+1})+\cdots \\ &+M_{r+1}(k_1,\dots,k_{r+1}-1) \\ \leq & b_{k_1-1,k}C_{r+1}(k_1-1,k,\dots,k_{r+1})\\ &+b_{k_1,k}\left(C_{r+1}(k_1,k-1,k,\dots,k_{r+1})+\cdots +C_{r+1}(k_1,\dots,k_{r+1}-1)\right)\\ < &b_{k_1,k}\left(C_{r+1}(k_1-1,k,\dots,k_{r+1})+\cdots+C_{r+1}(k_1,\dots,k_{r+1}-1) \right) \\ =& b_{k_1,k}C_{r+1}(k_1,k,k,k_4,\dots,k_{r+1}).
\end{align*} 

\noindent \textbf{Case 3:} $k_1>k_2>k_3$.

\noindent Suppose $k_1>k_2>k_3$. By the second conclusion of Lemma \ref{alglemma},  we have \begin{align*}& b_{k_1-1,k_2}C_r(k_1-1,k_2,\dots,k_r)+b_{k_1,k_2-1}C_r(k_1,k_2-1,\dots,k_r)\\=&b_{k_1,k_2}\left(C_r(k_1-1,k_2,\dots,k_r)+C_r(k_1,k_2-1,\dots,k_r) \right),  \end{align*}
and hence, applying property (iii) and the inductive hypothesis a final time, we have 
\begin{align*}M_{r+1}(k_1,\dots,k_{r+1})= &M_{r+1}(k_1-1,\dots,k_{r+1})+\cdots+M_{r+1}(k_1,\dots,k_{r+1}-1) \\ \leq & b_{k_1-1,k_2}C_{r+1}(k_1-1,\dots,k_{r+1}+b_{k_1,k_2-1}C_{r+1}(k_1,k_2-1,\dots,k_{r+1}) \\ &+b_{k_1,k_2}\left(C_{r+1}(k_1,k_2,k_3-1,\dots,k_{r+1})+\cdots+C_{r+1}(k_1,\dots,k_{r+1}-1) \right) \\ =& b_{k_1,k_2}\left(C_{r+1}(k_1-1,k,\dots,k_{r+1})+\cdots+C_{r+1}(k_1,\dots,k_{r+1}-1) \right) \\ =&b_{k_1,k_2}C_{r+1}(k_1,\dots,k_{r+1}),
\end{align*}
which completes the proof.
\end{proof}  

We now establish our less convenient, but exact, explicit formula for $M_r$. 

\begin{proof}[Proof of Theorem \ref{exact}] Fix integers $r\geq 2$ and $k_1,\dots,k_r\geq 3$. To compute an exact value of $M_r(k_1,\dots,k_r)$, one can repeatedly apply properties (ii) and (iii) from Definition \ref{wastedef} until reaching terms of the form $M_r(2,\dots,m,\dots,2)=m$ for some $3\leq m \leq k_j$, where the $j$-th coordinate is $m$ and all others are $2$, and then adding together all of the resulting values of $m$, with multiplicity. 

\noindent To count the number of occurrences of each terminating value, we count, for each ordered pair $(i,j)$ with $1\leq i,j\leq r$ and $i\neq j$, and each  $3\leq m \leq k_j$, the exact number of appearances of the term $M_r(2,\dots,3,\dots,m,\dots,2)$, where the $i$-th coordinate is $3$, the $j$-th coordinate is $m$, and all other coordinates are $2$. This count is relevant because every terminating value of $m$ descends from a unique term of this form, since the process would already terminate at the only other possible ``parent" $(2,\dots,m+1,\dots,2)$.  The number of applications of property (iii) required to reach such a term is $$(k_1-2)+\cdots+(k_i-3)+\cdots+(k_j-m)+\cdots+(k_r-2)=k_1+\cdots+k_r-2r-m+1,$$ and the summands on the left hand side above indicate exactly how many of those steps must be allocated to each coordinate. In other words, the desired count is precisely the multinomial coefficient ${k_1+\cdots+k_r-2r-m+1 \choose k_1-2,\dots,k_i-3,\dots,k_j-m,\dots,k_r-2}$. Then, for each occurrence of $M_r(2,\dots,3,\dots,m,\dots,2)$, we apply property (iii) a final time (ignoring all of the $2$'s as directed by property (ii)) to yield $$M_r(2,\dots,3,\dots,m,\dots,2)=m+M_r(2,\dots,3,\dots,m-1,\dots,2).$$ If $m>3$, the latter term is accounted for elsewhere, so we use this occurrence of $M_r(2,\dots,3,\dots,m,\dots,2)$ exclusively to account for this one terminating value of $m$. If $m=3$, then both terms on the right hand side above are terminating values of $3$, so the $3$ should be counted twice. Fortunately, our claimed formula has this double counting built in, as the two relevant coordinates can be chosen as $i$ and $j$ in either order.  Putting everything together, we have \begin{equation*}M_r(k_1,\dots,k_r)=\sum_{\substack{1\leq i,j\leq r \\ i\neq j}} \sum_{m=3}^{k_j} m {k_1+\cdots+k_r-2r-m+1 \choose k_1-2,\dots,k_i-3,\dots,k_j-m,\dots, k_r-2}, \end{equation*} as claimed. Finally, the simplified formula in the diagonal case comes from the fact that there are $r(r-1)$ ordered pairs $(i,j)$ with $1\leq i,j\leq r$ and $i\neq j$, and the multinomial coefficient is invariant under coordinate permutation. 
\end{proof}

Using the exact formula from Theorem \ref{exact}, we provide a convenient, explicit lower bound for $M_r$, which in the subsequent proof we show is asymptotically sharp. 

\begin{proof}[Proof of Corollary \ref{lbcor}] Fix integers $r\geq 2$ and $k_1,\dots,k_r\geq 3$. The desired lower bound can be seen from previous components of this paper in two ways: a simplified version of the double induction used to prove Theorem \ref{main} (where upper bounds are replaced by lower bounds and $b_{k_1,k_2}$ is replaced by the constant $3$), or an application of Theorem \ref{exact} which gives \begin{align*}M_r(k_1,\dots,k_r) &= \sum_{\substack{1\leq i,j\leq r \\ i\neq j}} \sum_{m=3}^{k_j} m {k_1+\cdots+k_r-2r-m+1 \choose k_1-2,\dots,k_i-3,\dots,k_j-m,\dots, k_r-2} \\ &\geq 3\sum_{\substack{1\leq i,j\leq r \\ i\neq j}} \sum_{m=3}^{k_j} {k_1+\cdots+k_r-2r-m+1 \choose k_1-2,\dots,k_i-3,\dots,k_j-m,\dots, k_r-2} \\ &=3{k_1+\cdots+k_r-2r \choose k_1-2,\dots,k_r-2}. \end{align*} The finally equality follows from the fact that every path (reducing single coordinates by $1$ at a time) from $(k_1,\dots,k_r)$ to $(2,\dots,2)$ passes through a unique point of the form $(2,\dots,3,\dots,m,\dots,2)$, with $3$ in coordinate $i$ and $m$ in coordinate $j$ and the remaining coordinates $2$, for some $3\leq m \leq k_j$. If $m>3$, the remainder of the path is determined, while if $m=3$ there are two choices, which as in the proof of Theorem \ref{main} is accounted for by the fact that the two relevant coordinates can be chosen as $i$ and $j$ in either order. 
\end{proof} 

\begin{proof}[Proof of Corollary \ref{mainasym3}] Fix integers $r\geq 2$ and $k\geq 3$. By Corollary \ref{lbcor}, it suffices for us to establish the desired asymptotic upper bound. By Theorem \ref{exact}, we have \begin{align*}M_r(k)&=r(r-1)\sum_{m=3}^k m {r(k-2)-m+1 \choose \underbrace{k-2,\dots, k-2}_{r-2},k-3,k-m} \\ &=r(r-1)\sum_{m=3}^k m\frac{(r(k-2)-m+1)!}{\left((k-2)!\right)^{r-2}(k-3)!(k-m)!} \\ &=\frac{(r(k-2))!}{\left((k-2)!\right)^r} \sum_{m=3}^k m(k-2) \frac{(k-2)!}{(k-m)!}\frac{r(r-1)}{r(k-2)(r(k-2)-1)\cdots (r(k-2)-m+2)} \\ &=\frac{(r(k-2))!}{\left((k-2)!\right)^r} \sum_{m=3}^k m \frac{(k-2)!}{(k-m)!}\frac{r-1}{(r(k-2)-1)\cdots (r(k-2)-m+2)}.\end{align*} We note that the rightmost denominator above consists of $m-2$ terms, each of which is at least $(r-1)(k-2)$, and the ratio $(k-2)!/(k-m)!$ is a product of $m-2$ terms, each of which is at most $k-2$, hence \begin{equation*}M_r(k)\leq \frac{(r(k-2))!}{\left((k-2)!\right)^r} \sum_{m=3}^k m\frac{(r-1)(k-2)^{m-2} }{((r-1)(k-2))^{m-2}}=\frac{(r(k-2))!}{\left((k-2)!\right)^r} \sum_{m=3}^k \frac{m}{(r-1)^{m-3}}. \end{equation*} Finally, the resulting summation satisfies \begin{align*}\sum_{m=3}^k \frac{m}{(r-1)^{m-3}}&=3+\sum_{j=1}^{k-3}\frac{3}{(r-1)^j}+\sum_{m=1}^{k-3}\frac{m}{(r-1)^m} \\ &\leq 3+\sum_{j=1}^{\infty}\frac{3}{(r-1)^j}+\sum_{m=1}^{\infty}\frac{m}{(r-1)^m} \\ &=3+\frac{3}{r-1}+\frac{r-1}{(r-2)^2}, \end{align*} where the infinite series are evaluated using the geometric series and its derivative. Since the resulting quantity is $3(1+o_{r\to\infty}(1))$, the asymptotic formula for $M_r(k)$ is established.
\end{proof}

\section{The waste function} \label{wastesec}

In this section, we verify our claims about the function $w_r$, namely Theorems \ref{wasterec} and \ref{wasteasym}. We begin with a separate proof of an exact formula for $w_r(3)$, implicit in Proposition \ref{r3}.

\begin{proposition}\label{waste3} For an integer $r\geq 2$,  \begin{equation*}w_r(3)=(3-e_r)r!-1. \end{equation*} 
\end{proposition}

\begin{proof}  We induct on $r$. When $r=2$, both sides of the formula are $0$, which establishes our base case. Now suppose the formula holds for a particular $r\geq 2$. By properties (ii), (iii), and (iv) in Definition \ref{wastedef}. we have 
\begin{align*} w_{r+1}(3)&=w_{r+1}(2,3,\dots,3)+w_{r+1}(3,2,3,\dots,3)+\cdots+w_{r+1}(2,3,\dots,3)+(r-2) \\ &= (r+1)w_{r}(3)+(r-1) \\ &=(r+1)[(3-e_r)r!-1]+(r-1) \\ &=(3-e_r)(r+1)!-1-1 \\ &=\left(3-e_r-\frac{1}{(r+1)!}\right)(r+1)!-1\\ &=(3-e_{r+1})(r+1)!-1,
\end{align*} 
as required.\end{proof}



We now establish an explicit lower bound for $w_r(k)$, proven by reducing to the case of $k=3$.

\begin{proof}[Proof of Theorem \ref{wasterec}] Fix integers $r\geq 2$ and $k\geq 4$. To estimate $w_r(k)$, we repeatedly apply property (iii) in Definition \ref{wasterec}. The first descension is straightforward, as \begin{align*}w_r(k)&=w_r(k,\dots,k) \\ &=w_r(k-1,k,\dots,k)+\cdots+w_r(k,\dots,k,k-1) + (r-2) \\ &=rw_r(k-1,k,\dots,k)+(r-2), \end{align*} where the last equality is property (iv). However, things quickly become more complicated. Once any coordinate is reduced to $2$, then in order to apply property (iii) again we must first apply property (ii) and eliminate that coordinate. Further, some of the newly introduced additive contributions in the next level will be $(r-3)$ instead of $(r-2)$, since we have reduced to $r-1$ coordinates. After each descension, we have a new combination of values of $w_{r-j}$ for some $j\geq 0$, which we informally call \textit{branches}, and some additive contributions of the form $(r-2-j)$ for some $j\geq 0$, which we informally refer to as \textit{accumulation}. To avoid ambiguity and double counting, and to take advantage of the exact formula yielded by Proposition \ref{waste3}, we primarily consider branches of the form $w_{r-j}(4,3,\dots,3)$ for $0\leq j \leq r-3$, specifically their contributions that do \textit{not} descend to branches of the form $w_{r-j-1}(4,3,\dots,3)$. We denote such contributions by $[w_r(k)]_j$.

\noindent In order to determine $[w_r(k)]_j$, we must count the number of relevant branches. Fixing $0\leq j \leq r-3$, we have ${r\choose j}$ choices of which coordinates to deplete all the way down to $2$, and hence eliminate with property (ii). To descend from $w_r(k)$ to $w_{r-j}(4,3,\dots,3)$ requires a total of $j(k-2)+(r-j)(k-3)-1=r(k-3)+j-1$ steps, of which $k-2$ must be assigned to each of the $j$ chosen coordinates. Once those choices are made, the remaining $(r-j)(k-3)-1$ steps must be distributed amongst the remaining $r-j$ coordinates, with $k-4$ steps assigned to one coordinate, for which there are $r-j$ choices, and $k-3$ steps assigned to all others, for a total of $$(r-j){r(k-3)+j-1 \choose \underbrace{k-2,\dots,k-2}_{j},\underbrace{k-3,\dots,k-3,k-4}_{r-j}}$$ possibilities. An application of properties (ii)-(iv) gives $$w_{r-j}(4,3,\dots,3)=w_{r-j}(3)+(r-j-1)w_{r-j-1}(4,3,\dots,3)+(r-2-j),$$ but the $w_{r-j-1}(4,3,\dots,3)$ branches will be accounted for in $[w_r(k)]_{j+1}$, so we have \begin{equation}\label{wrkj}[w_r(k)]_j={r\choose j}(r-j){r(k-3)+j-1 \choose \underbrace{k-2,\dots,k-2}_{j},\underbrace{k-3,\dots,k-3,k-4}_{r-j}}(r-j-2+w_{r-j}(3)). \end{equation}
 
\noindent We finish by noting that, so far, we have ignored all accumulation prior to branches of the form $w_{r-j}(4,3,\dots,3)$. Robustly accounting for such accumulation without double counting has proven difficult, but we can at least consider the family of branches that descend from $w_r(k)$ to $w_r(4,3,\dots,3)$ while also passing through $w_r(\ell)$ for each $4\leq \ell \leq k-1$. Specifically, \begin{align*}w_r(k)&=rw_r(k-1,k,\dots,k)+(r-2) \\ &\geq r(r-1)w_r(k-1,k-1,k,\dots,k)+(r-2)(1+r) \\ &\geq r(r-1)(r-2)w_r(k-1,k-1,k-1,\dots,k)+(r-2)(1+r+r(r-1)) \\ &\dots \\ &\geq r!w_r(k-1)+(r-2)r!(e_r-1) \\ &\dots \\ &\geq r!^{k-4}w_r(4)+(r-2)(e_r-1)r!(1+r!+r!^2+\cdots+r!^{k-5}) \\ &\dots \\ &\geq r!^{k-3}w_r(4,3,\dots,3)+(r-2)r!\left((e_r-1)(1+r!+r!^2+\cdots+r!^{k-5})+(e_r-2)r!^{k-4}\right).  \end{align*} 
While the resulting $w_r(4,3\dots,3)$ branches have already been accounted for in $[w_r(k)]_0$, the accumulation term on the right hand side has not. Simplifying this accumulation term and adding it to the previously considered contributions, we have  $$w_r(k)\geq (r-2)r!\left((e_r-1)\frac{r!^{k-3}-1}{r!-1}-r!^{k-4}\right)+\sum_{j=0}^{r-3}[w_r(k)]_j,$$ and the theorem now follows from \eqref{wrkj}.  
\end{proof} 

We now use Theorem \ref{wasterec} and Proposition \ref{waste3} to establish an asymptotic lower bound for $w_r(k)$ that has the same order of magnitude as $M_r(k)$. 

\begin{proof}[Proof of Theorem \ref{wasteasym}] Fix integers $r\geq 2$ and $k\geq 4$. Throughout the proof, $o(1)$ refers to decay as $r$ tends to infinity as opposed to $k$. By Theorem \ref{wasterec} and Proposition \ref{waste3}, ignoring the initial accumulation term prior to the summation and the $r-j-3\geq 0$ term within the summation, we have \begin{align*}w_r(k)&\geq \sum_{j=0}^{r-3}{r\choose j}{r(k-3)+j-1 \choose \underbrace{k-2,\dots,k-2}_{j},\underbrace{k-3,\dots,k-3,k-4}_{r-j}}(r-j)(3-e_{r-j})(r-j)! \\ &= \sum_{j=0}^{r-3}\frac{r!}{j!}\frac{(r(k-3)+j-1)!(k-3)}{\left((k-2)!\right)^{j}\left((k-3)!\right)^{r-j}}(r-j)(3-e_{r-j}) \\ &= \frac{(r(k-2))!}{\left((k-2)!\right)^r}\sum_{j=0}^{r-3}\frac{r!}{j!}\frac{(k-3)(k-2)^{r-j}(r(k-2))!}{(r(k-3)+j-1)!}(r-j)(3-e_{r-j}).\end{align*} Applying the uniform bound $3-e_{r-j}>3-e$, we see that Theorem \ref{wasteasym} will follow if we can establish \begin{equation}\label{goalsum}\sum_{j=0}^{r-3}a_j \geq \frac{1}{2}(1-o(1)), \end{equation} where $ a_j=\dfrac{r!}{j!}\dfrac{(k-3)(k-2)^{r-j}(r(k-3)+j)!(r(k-3)+j)}{(r(k-2))!(r(k-3)+j)}(r-j).$  For $j\geq r/10$, we have $j\to \infty$ as $r\to \infty$, hence Stirling's approximation $n!=\sqrt{2\pi n}(n/e)^n(1+o(1))$ applies to all four factorials in the definition of $a_j$. Therefore, bounding the resulting square root below by $1$ and canceling all factors of $e$, we have 
\begin{align*}a_j &= \sqrt{\frac{r^2(k-2)}{j(r(k-3)+j)}}\frac{(r/e)^r}{(j/e)^j}\frac{((r(k-3)+j)/e)^{r(k-3)+j}}{(r(k-2)/e)^{r(k-2)}}\frac{(k-2)^{r-j}(k-3)(r-j)}{r(k-3)+j}(1-o(1)) \\ &\geq \frac{r^r}{j^j}\frac{(r(k-3)+j)^{r(k-3)+j}}{(r(k-2))^{r(k-2)}}\frac{(k-2)^{r-j}(k-3)(r-j)}{r(k-3)+j}(1-o(1)), 
\end{align*}
provided $j\geq r/10$. Defining $\delta$ by $j=(1-\delta)r$, we see \begin{align*}a_j&\geq \frac{r^r}{((1-\delta)r)^{(1-\delta)r}} \frac{(r(k-2-\delta))^{r(k-2-\delta)}}{(r(k-2))^{r(k-2)}}\frac{(k-2)^{\delta r}(k-3)\delta r }{r(k-2-\delta)} (1-o(1)) \\ &= (1-\delta)^{-(1-\delta)r} \frac{(k-2-\delta)^{r(k-2-\delta)}}{(k-2)^{r(k-2)}}\frac{(k-2)^{\delta r}(k-3)\delta}{k-2-\delta}(1-o(1)) \\ &\geq\left(\frac{k-3}{k-2}\right)\delta(1-\delta)^{-(1-\delta)r} \left(1-\frac{\delta}{k-2}\right)^{r(k-2-\delta)}(1-o(1)) \\ &= \left(\frac{k-3}{k-2}\right)\delta(1-\delta)^{-(1-\delta) r}(1-\delta)^r\left(1-\frac{\delta}{k-2}\right)^{-\delta r}(1-o(1))\\ & =\left(\frac{k-3}{k-2}\right)\delta(1-\delta)^{\delta r}\left(1-\frac{\delta}{k-2}\right)^{-\delta r} (1-o(1)),\end{align*} provided $3/r\leq \delta \leq 9/10$, where the second to last line uses that $(1-\epsilon)^m\geq 1-m\epsilon$ for $\epsilon,m>0$. Defining $u$ by $\delta=u/\sqrt{r}$, we then have 
\begin{equation*}a_j \geq \left(\frac{k-3}{k-2}\right)\frac{u}{\sqrt{r}} \left(1-\frac{u}{\sqrt{r}}\right)^{u\sqrt{r}}\left(1-\frac{u}{(k-2)\sqrt{r}}\right)^{-u\sqrt{r}}(1-o(1)) 
\end{equation*} 
provided $3/\sqrt{r}\leq u \leq 9\sqrt{r}/10$. Since we are proving an asymptotic statement as $r\to \infty$, it is tempting to replace $\left(1-\frac{u}{\sqrt{r}}\right)^{u\sqrt{r}}$ with $e^{-u^2}(1+o(1))$, but since $u$ is not fixed with respect to $r$ we must give this statement its due consideration.  Using the Taylor approximation $\log(1-x)=-x+O(x^2)$ for $0\leq x \leq 1$, where $\log$ denotes the natural logarithm and $O(f(x))$ means less than a constant times $f(x)$, we see  \begin{equation*}\log \left( \frac{\left(1-\frac{u}{\sqrt{r}}\right)^{u\sqrt{r}}}{e^{-u^2}}\right)=u\sqrt{r}\log\left(1-\frac{u}{\sqrt{r}}\right)+u^2=u\sqrt{r}\left(-\frac{u}{\sqrt{r}}+O\left(\frac{u^2}{r}\right)\right)+u^2=O\left(\frac{u^3}{\sqrt{r}}\right).\end{equation*} The right hand side tends to $0$ as $r\to \infty$ if $u\leq r^{1/8}$, so the term inside the logarithm tends to $1$. The same reasoning applies to $\left(1-\frac{u}{(k-2)\sqrt{r}}\right)^{-u\sqrt{r}}=e^{u^2/(k-2)}(1+o(1))$.  Putting things together, we have \begin{equation*}a_j\geq \left(\frac{k-3}{k-2}\right) \frac{u}{\sqrt{r}} e^{-u^2\left(\frac{k-3}{k-2}\right)} (1-o(1)), \end{equation*} provided $3/\sqrt{r}\leq u \leq r^{1/8}$, which after backtracking through the substitutions yields \begin{equation*}\sum_{j=0}^{r-3} a_j \geq \left(\frac{k-3}{k-2}\right)\sum_{r-r^{5/8}\leq j \leq r-3} \left(1-\frac{j}{r}\right)e^{-(\sqrt{r}-j/\sqrt{r})^2\left(\frac{k-3}{k-2}\right)} (1-o(1)). \end{equation*} Further, we claim that the sum on the right hand side can be asymptotically bounded below by the corresponding integral, as \begin{align*}&\int_j^{j+1}\left(1-\frac{x}{r}\right)e^{-(\sqrt{r}-x/\sqrt{r})^2\left(\frac{k-3}{k-2}\right)} dx \\ =& \frac{1}{2}\int_{(\sqrt{r}-\frac{j+1}{\sqrt{r}})^2}^{(\sqrt{r}-\frac{j}{\sqrt{r}})^2} e^{-y\left(\frac{k-3}{k-2}\right)}dy \\ =&\frac{k-2}{2(k-3)}\left(e^{-\left(\sqrt{r}-\frac{j+1}{\sqrt{r}}\right)^2\left(\frac{k-3}{k-2}\right)}-e^{-\left(\sqrt{r}-\frac{j}{\sqrt{r}}\right)^2\left(\frac{k-3}{k-2}\right)}\right) \\=& \frac{k-2}{2(k-3)}e^{-\left(\sqrt{r}-\frac{j}{\sqrt{r}}\right)^2\left(\frac{k-3}{k-2}\right)}\left(e^{\left(2-\frac{2j+1}{r}\right)\left(\frac{k-3}{k-2}\right)}-1 \right) \\ \geq &  \frac{k-2}{2(k-3)}e^{-\left(\sqrt{r}-\frac{j}{\sqrt{r}}\right)^2\left(\frac{k-3}{k-2}\right)} \left(\frac{k-2}{k-3}\right)\left(2-\frac{2j}{r}-\frac{1}{r}\right) \\ =& \left(1-\frac{j}{r}\right)e^{-(\sqrt{r}-j/\sqrt{r})^2\left(\frac{k-3}{k-2}\right)} (1-o(1)), &\end{align*} where the second to last line uses that $e^x-1\geq x$ for $x\geq 0$. Therefore, \begin{equation*}\sum_{j=0}^{r-3} a_j \geq \left(\frac{k-3}{k-2}\right)\int_{r-r^{5/8}}^{r-3} \left(1-\frac{x}{r}\right)e^{-(\sqrt{r}-x/\sqrt{r})^2\left(\frac{k-3}{k-2}\right)} dx (1-o(1)). \end{equation*} Checking in with our goal, \eqref{goalsum}, and hence Theorem \ref{wasteasym}, will be established if we can show that the right hand side above is bounded below by $\frac{1}{2}(1-o(1))$. 

\noindent Indeed, returning to the substitution $y=(\sqrt{r}-x/\sqrt{r})^2$ and evaluating a convergent improper integral, we see \begin{align*}&\left(\frac{k-3}{k-2}\right)\int_{r-r^{5/8}}^{r-3} \left(1-\frac{x}{r}\right)e^{-(\sqrt{r}-x/\sqrt{r})^2\left(\frac{k-3}{k-2}\right)} dx \\ =& \frac{1}{2}\left(\frac{k-3}{k-2}\right) \int_{\frac{9}{r}}^{r^{1/4}}  e^{-y\left(\frac{k-3}{k-2} \right)}dy\\ = & \frac{1}{2}\left(\frac{k-3}{k-2}\right)\int_{0}^{\infty}  e^{-y\left(\frac{k-3}{k-2} \right)}dy (1-o(1)) \\ = & \frac{1}{2}(1-o(1)),\end{align*} and the theorem follows. \end{proof}

\section{Examples and computations} \label{compsec}

In this section, we display some collected numerical data in order to examine the efficiency of our estimates for $M_r(k)$ and $w_r(k)$, particularly in comparison to \eqref{cbound} and \eqref{jonib}, previously documented multinomial coefficient upper bounds for multicolor Ramsey numbers. The displayed values for $M_r(k)$ and $w_r(k)$ were computed via an algorithm coded in Python that enacts the properties listed in Definition \ref{wastedef}. The computations were conducted prior to our establishment of the exact formula for $M_r(k)$ provided in Theorem \ref{exact}, which reassuringly matches the collected data in all cases. For exact values of $w_r(k)$, the computational approach remains necessary. 

For clarity, we note that $M_r(k)-w_r(k)$ perfectly captures the upper bound for $R_r(k)$ available from Proposition \ref{recurs}. Further, with $\tilde{M}_r(k)$ and $\tilde{w}_r(k)$ as defined below, $\tilde{M}_r(k)-\tilde{w}_r(k)$ is the upper bound on $R_r(k)$ yielded by the convenient formula in Theorem \ref{main} and the lower bound in Theorem \ref{wasterec}, while $M_r(k)-\tilde{w}_r(k)$ is the upper bound on $R_r(k)$ yielded by Theorems \ref{exact} and \ref{wasterec}, which is the best explicit upper bound available from the results of this paper. The utilized notation and collected data are as follows.

\

\noindent $\tilde{M}_r(k)=\left(4-\frac{2}{k-1}\right)\frac{(r(k-2))!}{\left((k-2)!\right)^r}$, upper bound on $M_r(k)$ proved in Theorem \ref{main}

\noindent  $\tilde{w}_r(k):$ lower bound on $w_r(k)$ yielded by Proposition \ref{waste3} and Theorem \ref{wasterec}



\noindent $T_r(k)=\frac{(r(k-2)+2)!}{\left((k-1)!\right)^2\left((k-2)!\right)^{r-2}}$, upper bound for $R_r(k)$ as appears in \cite{joni}

\noindent  $C_r(k)=\frac{(r(k-1))!}{\left((k-1)!\right)^r}$, classical upper bound for $R_r(k)$

\begin{center}
\begin{table}[H]
\caption{ \ } 
\renewcommand{\arraystretch}{1.5}
\begin{tabular}{|| c | c || c | c || c | c ||  c | c ||}
\hline
\hline
$r$ & $k$ & $\tilde{M}_r(k)$ & $M_r(k)$ & $\tilde{w}_r(k)$ & $w_r(k)$ &  $T_r(k)$ & $C_r(k)$ \\
\Xhline{0.8pt}

$3$ & $4$ & $300$ & $288$ & $16$ & $16$ &  $560$ & $1680$ \\ 
\hline

$4$ & $4$ & $8400$ & $7920$ & $514$ & $554$ &  $25200$ & $369600$ \ \\ 
\hline

$5$ & $4$ & $378000$ & $352800$ & $24978$  & $26788$  & $1663200$ & $168168000$ \\ 
\hline

$3$ & $5$ & $5880$ & $5520$ & $214$ & $271$ & $11550$ & $34650$ \\ 
\hline

$4$ & $5$ & $1293600$ & $1182720$ & $60694$ & $75022$ &  $4204200$ & $63063000$ \\ 
\hline

$3$ & $6$ & $124740$ & $115500$ & $4644$ & $5248$ &  $252252$ & $756756$ \\ 
\hline

\end{tabular}
\end{table}
\end{center}

\section{Concluding remarks} \label{conc}

Throughout this paper, we have focused on establishing upper bounds for $R_r(k)$ stemming from the pigeonhole principle. While Proposition \ref{recurs} is one of few tools available to get upper bounds on multicolor Ramsey numbers for $r\geq 3$, there are some more refined estimates in certain cases. In particular, Fettes, Kramer, and Radziszowski showed that $R_4(3)\leq 62$, which beats the bound of $66$ yielded by Proposition \ref{r3}. Further, Eliahou \cite{eli}, building on work of Xu, Xie, and Chen \cite{Xu}, showed that \begin{equation}\label{elib} R_r(3)\leq \left(e-\frac{d}{24}\right)r!+1  \end{equation} for all $r\geq 4$, where $d=66-R_4(3)\geq 4$. 

One could ask if these improved estimates are in any way compatible with the results of this paper, and indeed they are. Specifically, rather than define the waste function $w_r(k)$ recursively in order to perfectly capture the strength of Proposition \ref{recurs}, we could have alternatively defined $$w_r(k)=M_r(k)-R_r(k)$$ to be the \textit{true} difference between the main term and the diagonal Ramsey number. Under this definition, properties (i) and (iii) for $w_r$ in Definition \ref{wastedef} become lower bounds rather than equalities, which only works in our favor, and the inequality portion of Proposition \ref{diff} no longer requires proof, as equality holds by definition. Theorem \ref{wasterec} holds just as before, and leads, along with Corollary \ref{mainasym3}, to the following stronger form of Theorem \ref{wasteasym}.

\begin{theorem}\label{compat} Suppose $r\geq 2$ and $k\geq 4$ are integers and $c>0$. If \begin{equation}\label{genhyp} R_r(3)\leq (3-c)r!(1+o(1)), \end{equation} then \begin{equation*}R_r(k)\leq \left(3-\frac{c}{2}\right)\frac{(r(k-2))!}{\left((k-2)!\right)^r}(1+o_{r\to \infty}(1)). \end{equation*}
\end{theorem}

\noindent The nominal modifications of the proof of Theorem \ref{wasteasym} needed to establish Theorem \ref{compat} are listed below:

\ 

\begin{itemize} \item Replace the constant $3-e$ with the constant $c>0$ from the hypothesis. \\ \item Reduce the upper limit of summation in \eqref{goalsum} to $r-3r^{1/4}$. This assures that $r-j\to \infty$ as $r\to \infty$, and allows for the invocation of \eqref{genhyp} without concern for small values of $r$. This raises the lower limit of integration on the resulting integral in $y$ from $9/r$ to $9/\sqrt{r}$, which still tends to $0$ as $r\to \infty$. \\ \item The remainder of the proof is identical. \\
\end{itemize}

\noindent Finally, we conclude our discussion with the following corollary of Theorem \ref{compat} and \eqref{elib}, which to our knowledge is the best-known asymptotic ($r\to \infty$) upper bound for diagonal Ramsey numbers for $k\geq 4$.

\ 

\begin{corollary}Let $d=66-R_4(3)\geq 4$. For integers $r\geq 2$ and $k\geq 4$, \begin{equation*}R_r(k)\leq \left(\frac{3+e}{2}-\frac{d}{48}\right)\frac{(r(k-2))!}{\left((k-2)!\right)^r}(1+o_{r\to \infty}(1)). \end{equation*} 
\end{corollary}

\

\noindent \textbf{Acknowledgements:}  This research was initiated during the Summer 2021 Kinnaird Institute Research Experience at Millsaps College. All authors were supported during the summer by the Kinnaird Endowment, gifted to the Millsaps College Department of Mathematics. At the time of submission, all authors except Alex Rice were Millsaps College undergraduate students.

\end{document}